\documentclass[pdflatex,sn-mathphys-num,alpha]{sn-jnl}
\usepackage{mathtools}
\usepackage{calrsfs}
\usepackage{microtype}
\usepackage{enumitem}
\usepackage{textcomp}
\usepackage{mathrsfs}%
\usepackage[title]{appendix}%
\usepackage{xcolor}%
\usepackage{textcomp}%
\usepackage{manyfoot}%
\usepackage{booktabs}%
\usepackage{amsfonts}
\usepackage{amsmath}
\usepackage{amssymb}
\usepackage{amstext}
\usepackage{amsthm}
\usepackage{comment}

\usepackage{hyperref} \usepackage{cleveref} \usepackage{subcaption}
\theoremstyle{plain}

\newtheorem{theorem}{Theorem}[section]

\newtheorem{lemma}[theorem]{Lemma}

\theoremstyle{definition} \newtheorem{definition}[theorem]{Definition}

\theoremstyle{remark} \newtheorem{remark}[theorem]{Remark}

\def\cB{\mathcal B}
\def\cF{\mathcal F}
 \def\cR{{\mathcal R}} \def\cC{\mathcal C}
 \def\cB{\mathcal B}
\newcommand{\cW}{\mathcal W} \newcommand{\cM}{\mathcal M}
  \def\cF{\mathcal F}

\newcommand{\tr}{\mathrm{Tr}} \newcommand{\PG}{\mathrm{PG}}
\newcommand{\AG}{\mathrm{AG}}

\newcommand{\cV}{\mathcal V}

\newcommand{\GF}[1]{\mathbb{F}_{#1}}

\begin{document}
	
	\title[Mutually intersecting varieties]{On mutually $\mu$-intersecting quasi-Hermitian varieties with some applications}

\author[1]{\fnm{Angela} \sur{Aguglia}}\email{angela.aguglia@poliba.it}
\author[2]{\fnm{Luca} \sur{Giuzzi}}\email{luca.giuzzi@unibs.it}
\author[1]{\fnm{Viola} \sur{Siconolfi}}\email{viola.siconolfi@poliba.it}

\affil[1]{\orgdiv{Dipartimento di Meccanica, Matematica e Management}, \orgname{Politecnico di Bari}, \orgaddress{\street{Via Orabona,4}, \city{Bari}, \postcode{70126}, \state{Italy}, \country{IT}}}

\affil[2]{\orgdiv{DICATAM}, \orgname{Universit\`a di Brescia},
\orgaddress{\street{Via Branze, 43}, \city{Brescia}, \postcode{25123}, \state{Italy}, \country{IT}}}

\abstract{
Let $\cW$ be a non-empty set of points of a finite Desarguesian projective space
$\PG(n,q)$.
A collection of varieties of $\PG(n,q)$  is mutually $\mu$-intersecting (relatively to $\cW$) if its elements meet all $\cW$ in the same number of points and pairwise intersect in $\cW$ in exactly $\mu$ points.

Here, we construct a new family of mutually $\mu$-intersecting algebraic varieties by using certain quasi-Hermitian varieties of $\PG(n, q^2)$, where $q$ is any prime power.
With the help of these quasi-Hermitian varieties we provide a new construction of $5$-dimensional
MDS codes over $\GF{q}$ as well as an infinite family of 
simple orthogonal arrays $OA(q^{2n-1},q^{2n-2},q,2)$ of index $\mu=q^{2n-3}$. 
}

\maketitle

%
%
\section{Introduction}
Let $M$ be a set of non-negative integers, $q$ a prime power and
$\cW$ a non-empty set of points of the Desarguesian projective space $\PG(n,q)$.
We say that a set of varieties $\cV(f_1),\dots,\cV(f_s)$ in  $\PG(n,q)$ are \emph{mutually $M$-intersecting relatively to $\cW$} if
the following properties hold:
\begin{itemize}
\item $|\cV(f_i)\cap\cW|=k$ for every $i=1,\dots,s$;
\item $|\cV(f_i)\cap \cV(f_j)\cap\cW|\in M$ for $i,j=1,\ldots,s$ and $i\neq j$.
\end{itemize}
This is a slight generalization of the original definition of $M$-intersecting varieties, 
\cite{FuMi0, FuMi3}, where  $\cW$ is the whole projective
space $\PG(n,q)$.
When $\cW$ is understood from the context and  $M=\{\mu\}$ we abuse the notation and simply talk about $\mu$-intersecting varieties
(instead of the more proper form $\{\mu\}$-intersecting). Having a 
family of $\mu$-intersecting varieties is directly related
to many combinatorial constructions, e.g.
 combinatorial designs, orthogonal arrays and maximum distance separable codes.

 In~\cite{AG} the authors
 consider a family of $\mu$-intersecting Hermitian varieties to obtain a family of orthogonal arrays. A similar approach is also followed in~\cite{AGMS}, where
an orthogonal array is obtained using the $\mu$-intersection property of some affine quasi-Hermitian varieties in the affine space  $\AG(3,q^2)$ with $q$ an even prime power.
The varieties of~\cite{AGMS} are the affine points of some BM quasi-Hermitian varieties
whose definition we briefly recall here (for more details, see Section~\ref{prelim}).

A non-singular Hermitian variety $H(n,q^2)$ in  $\PG(n,q^2)$, with $q$ a prime power, is the set of absolute points of a non-degenerate unitary polarity of $\PG(n,q^2)$.  Non-singular Hermitian varieties have been extensively studied both from the combinatorial and the algebraic point of view; see~\cite{Segre,Bose}. 
Their size is $(q^{n+1}+(-1)^n)(q^n-(-1)^n)/(q^2-1)$ and they are known to be $2$-character sets with respect to their intersections with hyperplanes. Specifically, their intersection with a non-tangent hyperplane has size
$(q^n+ (-1)^{n-1})(q^{n-1}-(-1)^{n-1})/(q^2-1)$, while their intersection with a tangent hyperplane contains $ 1 +q^2(q^{n-1}+(-1)^n)(q^{n-2}-(-1)^n)/(q^2-1)$ points; see~\cite{Segre} for further details. 

The combinatorial property of having just a few intersection numbers with hyperplanes is very strong (see e.g.~\cite{CK} for $2$-character sets), but it is not enough by itself to
characterize Hermitian varieties nor, more in general, the varieties
consisting of the absolute points of some polarity.
The notion of \emph{quasi-polar space}, see~\cite{SV22}, encompasses this combinatorial
property.
A quasi-polar space in $\PG(n,q)$ is a set of points 
whose intersection numbers with hyperplanes
are the same as
a non-degenerate classical polar space embedded in $\PG(n,q)$.


By Segre's theorem~\cite{Se55}, for $q$ odd and $n=2$, a quasi-conic (i.e. a $(q+1)$--arc)
is a conic. 
However, in general, proper quasi-quadrics and quasi-Hermitian varieties do exist.

In the case of unitary polarities, a set of points of $\PG(2,q^2)$ with the same size, i.e. $q^3+1$, and the same intersection numbers $\{1,q+1\}$ with lines as a Hermitian curve is called an \emph{unital}; it is well known that all Desarguesian projective planes $\PG(2,q^2)$ with $q>2$ contain non-classical unitals, i.e. unitals which are not projectively equivalent to the Hermitian curve; see~\cite{M79} and also~\cite{EB}.

Moving to higher dimension, a \emph{quasi-Hermitian variety} of $\PG(n,q^2)$ is the combinatorial abstraction of the non-degenerate Hermitian variety  defined as a
set of points in the
finite projective space $\PG(n,q^2)$ which has  the same intersection numbers with hyperplanes as the (projectively unique) non-degenerate Hermitian variety $H(n,q^2)$ of $\PG(n,q^2)$.
This implies that a quasi-Hermitian variety for $n>2$ also has the same number of points as $H(n,q^2)$, see~\cite{SV22}.
In particular, quasi-Hermitian varieties are always $2$-character sets;
see~\cite{CK}.

Obviously,  any  Hermitian  variety $H(n,q^2)$ is also a quasi–Hermitian variety; so we  call it the \emph{classical quasi–Hermitian  variety}  of $\PG(n,q^2)$. If $n=2$, a set with the same intersection characters with lines as a Hermitian curve has size either $q^3+1$ or $q^2+q+1$ and hence it is either a unital or a Baer subplane.


In 1976, F. Buekenhout introduced a general construction for unitals
in finite translation planes of square order~\cite{B}; later, R. Metz~\cite{M79}
showed the existence of non-classical
 unitals in $\PG(2,q^2)$ for $q>2$, that is, unitals which are not projectively equivalent to Hermitian curves. These unitals are
called Buekenhout-Metz (BM) unitals.
In~\cite{ACK}
a large family of quasi-Hermitian varieties of $\PG(n,q^2)$, indexed by pairs of elements of $\GF{q^2}$, is constructed.
These quasi-Hermitian varieties, called \emph{BM quasi-Hermitian varieties}, turn out to be a sort of generalization of
BM-unitals to higher dimension 
and their classification up to linear projectivities in the three-dimensional projective space is provided in ~\cite{AG3}, for $q$ odd and in \cite{AGMS}, for $q$ even.

In the present paper, we consider a $\mu$-intersecting family of BM quasi-Hermitian varieties and present some applications to codes and orthogonal arrays.
In Section~\ref{prelim} we recall some basics on the construction of BM quasi-Hermitian varieties and
fix the notation we shall use in the rest of the paper; in particular, in subsection~\ref{s2} we
recall the construction of BM quasi-Hermitian varieties, while in subsection~\ref{codes} we mention
some basics about coding theory; subsection~\ref{orth-intro} is dedicated to orthogonal arrays.

Our main geometric result is presented in Section~\ref{oarr}, where we describe a family of BM quasi-Hermitian
varieties mutually intersecting in $q^{2n-2}$ affine points in $\AG(n,q^2)$.
In Section~\ref{RS}, for the special case in which $q>4$ and $n=3$, 
we use these varieties
to obtain a $\GF{q}$-linear code 
with $q^5$ words, length $q$ and minimum distance $q-4$
which turns out to be equivalent to a subcode of the classical Reed-Solomon code over $\GF{q}$.
Finally, in Section~\ref{OA} we apply the construction of Section~\ref{oarr} to obtain an orthogonal array with parameters
$OA(q^{2n-1},q^{2n-2},q,2)$.

\section{Preliminaries}
\label{prelim}
\subsection{On BM quasi-Hermitian varieties}
\label{s2}
Let $q$ be a prime power (either odd or even, unless explicitly mentioned).
We consider the field extension $\GF{q^2}:\GF{q}$ and define the
(relative) trace and norm as
\[
\tr(x):=x+x^q,\qquad N(x):=x^{q+1}\qquad \text{ for all }x\in\GF{q^2}.
\]
Also, put $T_0:=\{x\in\GF{q^2}| \tr(x)=0\}$.

We shall work in the projective space $\PG(n,q^2)$, a
point of $\PG(n,q^2)$ will be described by a class of vectors in
homogeneous coordinates  $[(x_0,x_1,\ldots,x_n)]$.
We fix as hyperplane at infinity, the hyperplane $\Sigma_{\infty}$ of
$\PG(n,q^2)$ with equation $x_0=0$, i.e.
 \[
 \Sigma_{\infty}:=\{(0,x_1,\ldots,x_n): x_i\in \GF{q^2}\}\subset PG(n,q^2).
 \]
Let $\AG(n,q^2)$ be the
the affine space given by $\PG(n,q^2)\setminus\Sigma_{\infty}$.
Under this assumption, the points of $\AG(n,q^2)$ correspond to
classes
$[(x_0,x_1,\ldots,x_n)]\in\PG(n,q^2), x_0\neq 0$ and can be represented
by the vector coordinates $(x_1/x_0,\dots,x_n/x_0)\in\GF{q}^n$.

We shall use interchangeably affine and projective equations,
as the case warrants. 

Take $a\in\GF{q^2}$ and $b \in\GF{q^2} \setminus \GF{q}$. Define
$\mathcal{B}_{a,b}$  as the projective variety with equation
\begin{multline}
  \label{eq:bab}
  \cB_{a,b}: X_n^qX_0^q-X_nX_0^{2q-1}+a^q(X_1^{2q}+\ldots+X_{n-1}^{2q})-a(X_1^2+\ldots+X_n^2)X_0^{2q-2}=\\
  (b^q-b)(X_1^{q+1}+\ldots+X_{n-1}^{q+1})X_0^{q-1}
\end{multline}
and $\cF\subset \Sigma_{\infty}$ as the Hermitian cone
\[
\cF:=\{(0,x_1,\ldots,x_n)|x_1^{q+1}+\ldots+x_n^{q+1}=0\}.
\]
We can now define $\cM_{a,b}$ as
\[
\cM_{a,b}:=(\cB_{a,b}\setminus\Sigma_\infty)\cup \cF.
\]
We observe that 
$\cM_{a,b}$ is not described explicitly in terms of equations but
rather in terms of the gluing of the affine part of the  variety $\cB_{a,b}$
of degree $2q$ with some points at infinity.
In any case, as any set of points in $\PG(n,q^2)$ can be regarded as a (suitable) algebraic variety, this should not engender confusion.

If $a=0$, then $\cM_{a,b}$ is the point set of a classical quasi-Hermitian variety.  In the case in which $a\neq 0$ then, from \cite{ACK}*{Section 3} and \cite{AG2}*{Section 3}  we know that
$\cM_{a,b}$  is a non-classical  quasi-Hermitian variety if either of the following conditions holds:
\begin{enumerate}[label=(QH\arabic*)]
	\item\label{qh1} $n$ and $q$ are odd  and  $4a^{q+1}+(b^q-b)^2\neq 0$;
	\item\label{qh2} $n$ is even,  $q$ is odd  and  $4a^{q+1}+(b^q-b)^2$ is a non-square in $\GF{q}$;
	\item\label{qh3} $n$ and $q$ are even and $\tr(\frac{a^{q+1}}{(b^q+b)^2})=0$;
	\item\label{qh4} $n$ is odd and $q$ is even.
\end{enumerate}
This quasi-Hermitian variety is called a \emph{BM quasi-Hermitian variety}.

\subsection{Coding theory}
\label{codes}
In this paper we adopt the book by F. MacWilliams and N.Sloane~\cite{ECC} as a standard reference for coding theory.
We recall that a code $C$ of length $n$ over $\GF{q^s}$ with $s\geq1$ and minimum (Hamming) distance
$d$ is just a set of $N$ vectors of $\GF{q^s}^n$ which pairwise differ in at least $d$ coordinates.
A code $C$ is $\GF{q}$-linear if it is a vector space over $\GF{q}$, clearly in this case $N=q^k$ where $k=\dim_{\GF{q}}(C)$.

In this paper we shall consider a code whose words are in $\GF{q^2}^n$, but which is only $\GF{q}$-linear. 
In particular, we shall adopt the following definition.
\begin{definition}
Let $C\subseteq\GF{q^2}^n$ be a code which is $\GF{q}$-linear. We say that $C$ is
\emph{equivalent} to an $\GF{q}$-linear code $C'\subseteq\GF{q^2}^n$ if there is an $\GF{q}$-linear isometric mapping 
$\varphi:C\to C'$.
\end{definition}
 By the Singleton bound for $\GF{q}$-linear codes, $N\leq q^{n-d+1}$.
 When equality holds, a code is \emph{maximum distance separable}  or MDS.
 These codes play a relevant role in the theory of error correcting codes. If a code is MDS the above equality can be read also as $d=n-k+1$.

Let $\GF{q}[t]$ be the ring of all polynomials in the indeterminate
$t$ with coefficients in $\GF{q}$. We write $\GF{q}[t]_{\deg\leq k-1}$
for the vector space of all polynomials over $\GF{q}$ in the
unknown $t$ and degree at most $k-1$. Clearly $\dim(\GF{q}[t]_{\deg\leq k-1})=k$.

A well known family of linear MDS codes is given by  Reed-Solomon codes (see \cite{RS,SR} notation: $RS_q(n,k)$), they can be defined as follows:
\begin{definition}The Reed-Solomon code over $\GF{q}$ is
    \[
RS_q(n,k)=\{(f(\omega_1),\ldots,f(\omega_n)):f\in \GF{q}[x]_{\deg\leq k-1}\}
    \]
    where $n\leq q-1$ and $\omega_1,\ldots,\omega_n$ are distinct elements of $\GF{q}^*$.
    If we take all values of $\omega_i\in\GF{q}$, then we have
    a code of length $q$ which is called an
    \emph{extended Reed-Solomon code} and denoted by
    $RS_q(q,k)$.
\end{definition}
Extended Reed-Solomon codes can be further extended by adding an 
extra symbol to each codeword (the component corresponding 
to evaluating the polynomials ``at infinity''). Codes thus
obtained are called \emph{doubly extended Reed-Solomon}.
We shall also denote them by the symbol $RS_q(q+1,k)$.

It is straightforward to see that the parameters
of $RS_q(n,k)$ are $[n,k,n-k+1]$.

A generalized (extended/doubly extended) Reed-Solomon code $GRS_q(n,k)$ can be obtained from
a Reed-Solomon code by multiplying each entry in every codeword by
some fixed non-zero elements of $\GF{q}$. 
In particular,
\begin{definition} A generalized Reed-Solomon code is 
    \[
GRS_q(n,k)=\{(\alpha_1 f(\omega_1),\ldots,\alpha_n f(\omega_n)):f\in \GF{q}[x]_{\deg\leq k-1}\}
    \]
    where $n\leq q+1$, $\omega_1,\ldots,\omega_n$ are distinct elements of $\GF{q}\cup\{\infty\}$ and $\alpha_1,\dots,\alpha_n\in\GF{q}^*$.
\end{definition}
It is a long-standing conjecture that when $k>3$, every $q$-ary MDS code
of length $q+1$ and dimension $k$ must be a generalized doubly extended Reed-Solomon code. In several cases (e.g. when $q$ is prime or when $k$ is small) the conjecture has been settled; see~\cite{BL}.

\subsection{Orthogonal arrays}
\label{orth-intro}

An orthogonal array of strength $t$ is a matrix $OA(N,k,v,t)$  of type  $N\times k$  with the following properties:
\begin{itemize}
    \item the entries  are in a set $S$ of cardinality $v$, $v$ is called the number of `levels';
 \item every $t$-uple of elements in $S$ appears exactly $N/v^t$ times in each subarray  of type $N\times t$.
    The number $N/v^t$   is called the \emph{index} of the orthogonal array.
\end{itemize}

Orthogonal arrays are versatile objects that since their introduction (see\cite{Rao1, Rao2}) have been used for a wide range of applications: from  statistics to cryptography, we mention in particular a recent work on orthogonal arrays and drones \cite{UAV}.
We refer to~\cite{Sloane} for further details.

In what follows we introduce a well-known geometric method to build an orthogonal array. We start with a set of $k$ homogeneous forms $f_1,\ldots, f_k$ in $n + 1$ unknowns over $\GF{q}$, we denote by $V(f_i)$ the variety associated to $f_i$. We then consider a set $\cW\subset \GF{q}^{n+1}$ with the crucial property that $|V(f_i)\cap V(f_j)\cap \cW|$ does not depend on $i$ and $j$, we denote by $N=|\cW|$. This allows the definition of a $N\times k$ array:
\[
A(f_1,\ldots,f_k)=\left\{
(
f_1(x) \ldots
f_k(x)):x\in\cW
\right\}.
\]
which turns out to be an orthogonal array.
This method has been used for Hermitian varieties in \cite{AG} and for quasi-Hermitian varieties in \cite{AGMS} in the particular case of even characteristic and $n=3$. In Section \ref{OA} of the present paper we extend the latter result removing the constraint on $n$ and on the characteristic.

\section{$\mu$-intersecting quasi-Hermitian varieties}\label{sec:quasiH_varieties}
\label{oarr}

As before, let $\Sigma_{\infty}=\{(0,x_1,x_2,\ldots,x_n)|x_i\in \GF{q^2}\}$ be the hyperplane at
infinity of $\PG(n,q^2)$.
Denote by $G$ the subgroup of
$\mathrm{PGL}(n+1,q^2)$ consisting of all collineations
\[ (x_0,x_1,\dots,x_n)\to (x_0,x_1,\dots,x_n)M \]
admitting a (normalized) matrix representative of the form
\begin{equation} \label{collin}
M=\begin{pmatrix}
    1 &\alpha_1& \alpha_2 &\ldots& \alpha_{n-1} & \alpha_n\\
    0 & 1& 0 &\ldots& 0 & \beta_1\\
       0& 0& 1 &\ldots & 0 & \beta_2\\
        \vdots&&&\ddots&&\vdots\\
        0& 0& 0 & \ldots &1 & \beta_{n-1}\\
    0 &0& 0 &\ldots& 0 & 1 \\
  \end{pmatrix},
\end{equation}
with $\alpha_s, \beta_{
j}\in \GF{q^2}$. The group $G$ has order
$q^{2(2n-1)}$. It stabilizes the hyperplane $\Sigma_\infty$, fixes
the point $P_{\infty}(0,0,\ldots,0,1)$ and acts transitively on
$\mathrm{AG}(n,q^2)$.

From now on, we will always assume the following to hold:  $a\in\GF{q^2}^*$,  $b \in\GF{q^2}\setminus\GF{q}$ and one of the conditions \ref{qh1}...\ref{qh4} holds.

Consider the affine equation of $\cB_{a,b}$:
\begin{equation}
 X_n^q-X_n+a^q(X_1^{2q}+\ldots+X_{n-1}^{2q})-a(X_1^{2}+\ldots+X_{n-1}^{2})=
	(b^q-b)(X_1^{q+1}+\ldots+X_{n-1}^{q+1})
\end{equation}

The subgroup $\Psi$ of $G$ preserving $\cB_{a,b}$ consists of all
collineations whose matrices satisfy conditions
\[\begin{cases}
  \begin{aligned} \alpha_n^q-\alpha_n+a^q(\alpha_{1}^{2q}+\dots+\alpha_{n-1}^{2q})-a(\alpha_{1}^2+\dots+\alpha_{n-1}^2)=
   (b^q-b)(\alpha_{1}^{q+1}+\dots+\alpha_{n-1}^{q+1})
    \end{aligned} \\
          \beta_1=(b-b^q)\alpha_1^q-2a \alpha_1\\
     \beta_2=(b-b^q)\alpha_2^q-2a\alpha_2\\
     \qquad\qquad\vdots\\
     \beta_{n-1}=(b-b^q)\alpha_{n-1}^q-2a\alpha_{n-1}\\
   \end{cases}. \]
Thus, $\Psi$ contains  $q^{(2n-1)}$ collineations and acts on the
affine points of $\cB_{a,b}$ as a sharply transitive permutation group.

The stabilizer in $G$ of the origin $O(1,0,0,\ldots,0)$  fixes the
line $OP_{\infty}$ pointwise, while acting transitively on the set of points distinct  from $P_{\infty}$ contained
 in any other line passing through $P_{\infty}$. Furthermore, the
center of $G$ comprises all collineations induced by
\begin{equation} \label{centre}
	\begin{pmatrix}
		1 &0& \ldots &0& \alpha_n\\
		0 & 1& \ldots&0& 0\\
 \vdots&&&&\vdots\\
		0& 0&\ldots& 1 & 0\\
		0 &0& \ldots&0 & 1 \\
	\end{pmatrix},
\end{equation}
with $\alpha_n \in \GF{q^2}$. The subset of $G$ consisting of all
collineations induced by \eqref{centre}, with $\alpha_n\in \GF{q}$, is a
normal subgroup $N$ of $G$ that acts semiregularly on the affine
points of $\mathrm{AG}(n,q^2)$ and preserves each line parallel to 
$X_n=0$.

Let $\mathcal{C}=\{a_1=0,\ldots,a_q\}$ be a transversal
of $\GF{q}$, viewed as an additive subgroup of $\GF{q^2}$.
Observe that if $\varepsilon$ is a primitive element of
$\GF{q^2}$, then $\mathcal{C}$ can be taken as
\[ \mathcal{C}=\varepsilon\GF{q}=\{ \varepsilon w | w\in\GF{q} \}. \]
Let also
 $\cR$ denote the subset of $G$
 whose
collineations are induced by matrices of the form
\begin{equation}\label{rep}
 M'= \begin{pmatrix}
    1 &\alpha_1&\ldots&\alpha_{n-2}&\alpha_{n-1}&\alpha_n\\
    0 & 1&\ldots & 0 & 0& 0\\
    \vdots&&&&&\vdots\\

    0& 0&\ldots &0 & 1 & 0\\
  0 &0& \ldots&0 &0&1 \\
  \end{pmatrix},
\end{equation}
where $\alpha_1,\dots,\alpha_{n-1}\in \GF{q^2}$ are taken arbitrarily,
while $\alpha_n$ is the unique solution in
$\mathcal{C}$ of the equation
\begin{equation}\label{ara1}
\alpha_n^q-\alpha_n+a^q(\alpha_{1}^{2q}+\ldots+\alpha_{n-1}^{2q})-a(\alpha_{1}^2+\ldots+\alpha_{n-1}^2)=
     (b^q-b)(\alpha_{1}^{q+1}+\ldots+\alpha_{n-1}^{q+1}).
\end{equation}
The set  $\cR$ has cardinality $q^{2n-2}$. Since
 $\cM_{a,b}\cap \AG(3,q^2)=\cB_{a,b}\cap \AG(3,q^2)$, the set $\cR$
 can be used to
have a collection of  
BM quasi-Hermitian varieties pairwise
intersecting in $q^{2n-2}$ affine points.
The construction is as follows.
Denote by
\[\begin{split} F(X_0, X_1,\dots, X_n)=
  X_0^q X_n^q-X_nX_0^{2q-1}+a^q(X_1^{2q}+\ldots+X_{n-1}^{2q})+\\
  -a(X_1^{2}+\ldots+X_{n-1}^{2})X_0^{2q-2}
  -(b^q-b)(X_1^{q+1}+\ldots+X_{n-1}^{q+1})X_0^{q-1}
\end{split}
\] the homogeneous form associated to the variety $\cB_{a,b}$.

We observe that if $ g, g' \in \mathcal{R}$  with $ g \neq g' $, then $ g{g'}^{-1} \notin \Psi $; equivalently, $ \mathcal{V}(F^g) \neq \mathcal{V}(F^{g'}) $.

\begin{theorem}
  The  set of  forms $\{F^g| g\in \cR\}$
  is associated to a set $\{ \cV(F^g)| g\in\cR\}$ of $q^{2n-2}$
  varieties in $\PG(n,q^2)$,  mutually
  intersecting in $q^{2n-2}$ affine points.
\end{theorem}	
\begin{proof}
The set \( \{ \mathcal{V}(F^g) \mid g \in \mathcal{R} \} \) contains \( q^{2n-2} \) distinct varieties as \( \mathcal{V}(F^{g}) \neq \mathcal{V}(F^{g'}) \) whenever \( g \neq g' \).
{
We count the cardinality of the intersection of  $\mathcal{V}(F^g)$ and $\mathcal{V}(F^{g'})$, $g\neq g'$, by studying the solutions of the following system in $\mathcal{A}:=\{(x_0,x_1,\ldots,x_{n})\in\GF{q^2}^{n+1}:
x_{0}=1\}$: 
\begin{equation}
	\left\{\begin{array}{l}
		\label{orto11}
		F^g(X_0,X_1,X_2,\ldots,X_n)=0 \\
		F^{g'}(X_0,X_1,X_2,\ldots,X_n)=0
	\end{array}\right.
\end{equation}
Substituting the equations for $F^{g}$ and $F^{g'}$ we obtain
\begin{equation}
	\left\{\begin{array}{l}
          \label{orto22}
          \begin{aligned}
            X_n^q-&X_n+a^q(X_1^{2q}+\dots+X_{n-1}^{2q})-a(X_1^2+\dots+X_{n-1}^2)+
            \\ &-(b^q-b)(X_1^{q+1}+\dots+X_{n-1}^{q+1})+  
            [2a^q\alpha_1^q-(b^q-b)\alpha_1]X_1^q +\dots+\\	& [2a^q\alpha_{n-1}^q-(b^q-b)\alpha_{n-1}]X_{n-1}^q
            -[2a\alpha_1+(b^q-b)\alpha_1^q]X_1+\\ &\ldots-[2a\alpha_{n-1}+(b^q-b)\alpha_{n-1}^q]X_{n-1}=0
	\end{aligned}	\\
          \begin{aligned}
            X_n^q-&X_n+a^q(X_1^{2q}+\dots+X_{n-1}^{2q})-a(X_1^2+\dots+X_{n-1}^2)+
            \\ &-(b^q-b)(X_1^{q+1}+\dots+X_{n-1}^{q+1})+  
            [2a^q{\alpha'_1}^q-(b^q-b)\alpha'_1]X_1^q +\dots+\\	& [2a^q{\alpha'_{n-1}}^q-(b^q-b){\alpha'_{n-1}}]X_{n-1}^q
            -[2a{\alpha'_1}+(b^q-b){\alpha'_1}^q]X_1+\\ &\ldots-[2a{\alpha'}_{n-1}+(b^q-b){\alpha'_{n-1}}^q]X_{n-1}=0,
	\end{aligned}	
	\end{array}\right.
\end{equation}
where $(\alpha_1,\ldots,\alpha_n)$ and $(\alpha'_1,\ldots,\alpha'_n)$ are the coefficients appearing in the matrices associated to $g$ and $g'$ respectively.
Subtracting the first equation from the second we get
\begin{equation}\label{trc:22}
	(s_1X_1+\ldots+s_{n-1}X_{n-1})^q-(s_1X_1+\ldots+s_{n-1}X_{n-1})=0,
\end{equation}
where $s_i=2a(\alpha_i-\alpha'_i)+(b^q-b)(\alpha_i^q-{\alpha'_i}^q)$ for every $i=1,\ldots,n-1$. Since $g\neq g'$ at least some of the $s_i$ are non zero.
 Hence, Equation~\eqref{trc:22} is equivalent to the union
of $q$ linear equations
in $X_1, \ldots, X_{n-1}$ over $\GF{q^2}$. Thus, there are
$q^{2n-3}$ tuples $(X_1, \ldots, X_{n-1})$ satisfying \eqref{trc:22}.
For each such a tuple, \eqref{orto22} has $q$ solutions in $X_n$ that
provide a coset of $\GF{q}$ in $\GF{q^2}$. Therefore, the system
\eqref{orto11} has $q^{2n-2}$ solutions in $\mathcal{A}$ and the result
follows.
}

\end{proof}
The following lemma will be used in the forthcoming sections. 
\begin{lemma}
\label{different} Let $\cW:=\{(x_0,x_1,\ldots,x_n)|x_0=1,x_n \in \mathcal{C}\}$.
If $(1,x_1,\dots,x_n)$ and $(1,x_1',\dots,x_n')$ are two distinct elements of $\cW$, then there exists $g\in{\mathcal R}$ such that
\[ F^g(1,x_1,\dots,x_n)\neq F^g(1,x'_1,\dots,x'_n). \]
\end{lemma}
\begin{proof}
Assume by contradiction that
\begin{equation}\label{eq:repeated_rows}
F^g(1,x_1,\dots,x_n)=F^g(1,x'_1,\dots,x'_n)\;\; \forall g\in \mathcal{R}
\end{equation}
for $(1,x_1,\dots,x_n)$ and $(1,x'_1,\dots,x'_n)$ two distinct (n+1)-tuples in $\mathcal{W}$. Equality \eqref{eq:repeated_rows} implies $q^{2n-2}$ equations in the $2n$ variables $x_1,\dots,x_n,x'_1,\dots,x'_n$.
We write \eqref{eq:repeated_rows} explicitly:
\begin{multline*}
x_{n}^q-x_n+a^q\left(\sum_{i=1}^{n-1} x_i^{2q}+2\alpha_i^qx_i^q\right)-a\left(\sum_{i=1}^{n-1} x_i^2+2\alpha_ix_i\right)\\-(b-b^q)\left(\sum_{i=1}^{n-1}x_i^{q+1}+\alpha_i x_i^q+\alpha_i^qx_i\right)=
\\
={x'^q}_{n}-x'_n+a^q\left(\sum_{i=1}^{n-1}{x'^{2q}}_i+2\alpha_{i}^q{x'_{i}}^q\right)-a\left(\sum_{i=1}^{n-1}{x'_i}^2+2\alpha_ix'_i\right)-\\(b-b^q)\left(\sum_{i=1}^{n-1}{x'_i}^{q+1}+\alpha_i {x'_i}^q+\alpha_i^qx'_i\right).
\end{multline*}
Or equivalently:
\begin{multline}
  \label{eq_alpha}
  (x_n-x'_n)^q-(x_n-x'_n)+a^q\left(\sum_{i=1}^{n-1} (x_i^2-{x'_i}^2)^{q}\right)-a\left(\sum_{i=1}^{n-1} (x_i^2-(x'_i)^2)\right)\\
  -(b-b^q)\left(\sum_{i=1}^{n-1}x_i^{q+1}-{x'_i}^{q+1}\right)=\\
=a^q\left(\sum_{i=1}^{n-1}2\alpha_i^q(x_i-x_i')^q\right)-a\left(\sum_{i=1}^n2\alpha_i(x_i-x'_i)\right)+(b-b^q)\left(\sum_{i=1}^{n-1}\alpha_i(x_i-x'_i)^q+\alpha_i^q(x_i-x'_i)\right)
\end{multline}
for every possible value of $\alpha_1,\dots,\alpha_{n-1}$. Noting that the left hand side does not depend on the $\alpha_i$ and it is $0$ for $\alpha_1=\dots=\alpha_n=0$, we deduce that 
\[
a^q\left(\sum_{i=1}^n2\alpha_i^q(x_i-x_i')^q\right)-a\left(\sum_{i=1}^n2\alpha_i(x_i-x'_i)\right)+(b-b^q)\left(\sum_{i=1}^n\alpha_i(x_i-x'_i)^q+\alpha_i^q(x_i-x'_i)\right)=0,
\]
where the last equality holds for every possible value of $\alpha_1,\dots,\alpha_{n-1}$.
In particular, one can fix $\alpha_2=\dots=\alpha_{n-1}=0$ and obtain:
\[
B^q-B=0, \text{ where }B=(\alpha_1^q(b-b^q)-2\alpha_1a)(x_1-x_1')
\]
this means that $B\in \GF{q}$ for every possible value of $\alpha_1\in \GF{q^2}$; note that this is only possible if $x_1=x'_1$. The same reasoning holds in general, thus we conclude that $x_i=x_i'$ for every $i=1,\dots,n-1$. We deduce that the first equation from \eqref{eq_alpha} reads:
\[
(x_n-x'_n)^q-(x_n-x_n')=0,
\]
implying $x_n-x'_n\in\GF{q}$, because $x_n,x'_{n}\in{\mathcal{C}}$ we deduce that 
$x_n=x_n'$. We  have obtained $(1,x_1\dots,x_n)=(1,x'_1,\dots,x'_n)$, a contradiction that concludes the proof.
\end{proof}

\section{A construction of $q$-ary MDS codes}
\label{RS}
Let $G_1(x),\dots,G_N(x)$ be $N$ multivariate polynomials over $\GF{q}$.
The evaluation code $C:=C(G_1,\dots,G_N;\cW)$ defined by $G_1,\dots,G_N$ over a set $\cW$
is the image of the map
\begin{align*}ev_{G_1,\dots,G_N}:&
  \cW \to \GF{q}^N\\
  &x \to (G_1(x),\dots,G_N(x)).
 \end{align*}
Assume that $C$ has $q^t$ codewords and Hamming distance $d$. By the Singleton bound, $q^t \leq q^{N-d+1}$ and codes attaining the bound are MDS codes.
 If $C$ attains the Singleton bound  then the restrictions of all codewords to any given $ t=N-d+1$ places must all be different, namely in any $t$ positions all possible  vectors occur exactly once. This means that a necessary condition for $C$ to be MDS is that any $t$ of the varieties $V(G_i)$, for $i=1, \ldots, N$ meet in exactly one point in $\cW$. Here $V(G_i)$ is the algebraic variety associated to the form $G_i$.

Set $\cW:=\GF{q^2}\times \GF{q^2}\times \cC$ where $\cC$ is a transversal of $\GF{q}$ viewed as an additive subgroup of $\GF{q^2}$. The size of $\cW$ is $q^5$. Consider the following subset $\Omega $ of $\GF{q^2}^2$ such that for each
$(\omega_1^i, \omega_2^i )\in \Omega$ with $i\geq 5$, the following 
condition holds
\begin{equation}\label{luc1}
\begingroup 
\setlength\arraycolsep{3pt}
 \det\begin{pmatrix}
   1 & \omega_1^1 &\omega_2^1&  (\omega_1^1)^q & (\omega_2^1)^q  \\
   1 & \omega_1^2 &\omega_2^2&  (\omega_1^2)^q & (\omega_2^2)^q \\
   1 & \omega_1^3 &\omega_2^3&  (\omega_1^3)^q & (\omega_2^3)^q \\
   1 & \omega_1^4 &\omega_2^4&  (\omega_1^4)^q & (\omega_2^4)^q \\
    1 & \omega_1^5 &\omega_2^5&  (\omega_1^5)^q & (\omega_2^5)^q
  \end{pmatrix} \neq 0.
  \endgroup
  \end{equation}
  Assume first that $q$ is odd; then Condition~\eqref{luc1} implies
  \begin{equation}\label{luc11}
 \begingroup 
\setlength\arraycolsep{5pt}
 \det\begin{pmatrix}
   1 & \omega_1^1+(\omega_1^1)^q &\omega_2^1+(\omega_2^1)^q&\omega_1^1-(\omega_1^1)^q & \omega_2^1-(\omega_2^1)^q  \\
   1 & \omega_1^2+(\omega_1^2)^q &\omega_2^2+(\omega_2^2)^q& \omega_1^2 -(\omega_1^2)^q & \omega_2^2-(\omega_2^2)^q \\
   1 & \omega_1^3+(\omega_1^3)^q &\omega_2^3+(\omega_2^3)^q&  \omega_1^3-(\omega_1^3)^q & \omega_2^3-(\omega_2^3)^q \\
   1 & \omega_1^4+(\omega_1^4)^q &\omega_2^2+(\omega_2^4)^q&  \omega_1^4-(\omega_1^4)^q & \omega_2^4-(\omega_2^4)^q \\
    1 & \omega_1^5+(\omega_1^5)^q &\omega_2^5+(\omega_2^5)^q&  \omega_1^5-(\omega_1^5)^q & \omega_2^5-(\omega_2^5)^q
  \end{pmatrix} \neq 0
  \endgroup.
  \end{equation}

 Fix a basis $(1,\epsilon)$ of $\GF{q^2}$  regarded as  a vector space over $\GF{q}$, with $\epsilon$  a primitive element of $\GF{q^2}$ such that $\epsilon^q+\epsilon=0$.
Write $\omega_i^j=\omega_{i,0}^j+\omega_{i,1}^j \epsilon$ for all $i=1,2$ and $j=1,\ldots,5$.
Then \eqref{luc11} becomes
\begin{equation}
\begingroup 
\setlength\arraycolsep{3pt}
  \label{luc11b} \det\begin{pmatrix}
   1 & \omega_{1,0}^1 &\omega_{2,0}^1&  \omega_{1,1}^1 & \omega_{2,1}^1  \\
   1 & \omega_{1,0}^2 &\omega_{2,0}^2&  \omega_{1,1}^2 & \omega_{2,1}^2 \\
   1 & \omega_{1,0}^3 &\omega_{2,0}^3&  \omega_{1,1}^3 & \omega_{2,1}^3 \\
   1 & \omega_{1,0}^4 &\omega_{2,0}^4&  \omega_{1,1}^4& \omega_{2,1}^4 \\
    1 & \omega_{1,0}^5 &\omega_{2,0}^5&  \omega_{1,1}^5 & \omega_{2,1}^5
  \end{pmatrix} \neq 0
  \endgroup
\end{equation}
for any choice of five elements in $\Omega$.
Observe that all entries in~\eqref{luc11b} are over $\GF{q}$; also
Condition~\eqref{luc11b} states that
the possible rows of that matrix are the coordinates of the points of an
arc in $\AG(4,q)$.
Consequently, $S=|\Omega|\leq q$ for $q$ odd.

Consider now the case in which $q$ is even. It is always possible to fix a basis $(1,\epsilon)$ of $\GF{q^2}$   over $\GF{q}$ by taking $\epsilon\in \GF{q^2}\setminus \GF{q}$ such that $\epsilon^q=1+\epsilon$. Then we notice that Condition~\eqref{luc1} implies

  \begin{equation}\label{luc11c}
  \begingroup 
\setlength\arraycolsep{5pt}
 \det\begin{pmatrix}
   1 & \omega_1^1+(\omega_1^1)^q &\omega_2^1+(\omega_2^1)^q&\epsilon^q\omega_1^1+\epsilon(\omega_1^1)^q & \epsilon^q\omega_2^1+\epsilon(\omega_2^1)^q  \\
   1 & \omega_1^2+(\omega_1^2)^q &\omega_2^2+(\omega_2^2)^q& \epsilon^q\omega_1^2 +\epsilon(\omega_1^2)^q & \epsilon^q\omega_2^2+\epsilon(\omega_2^2)^q \\
   1 & \omega_1^3+(\omega_1^3)^q &\omega_2^3+(\omega_2^3)^q&  \epsilon^q\omega_1^3+\epsilon(\omega_1^3)^q & \epsilon^q\omega_2^3+\epsilon(\omega_2^3)^q \\
   1 & \omega_1^4+(\omega_1^4)^q &\omega_2^2+(\omega_2^4)^q&  \epsilon^q\omega_1^4+\epsilon(\omega_1^4)^q & \epsilon\omega_2^4+\epsilon(\omega_2^4)^q \\
    1 & \omega_1^5+(\omega_1^5)^q &\omega_2^5+(\omega_2^5)^q&  \epsilon^q\omega_1^5+\epsilon(\omega_1^5)^q & \epsilon^q\omega_2^5+\epsilon(\omega_2^5)^q
  \end{pmatrix} \neq 0.
  \endgroup
  \end{equation}
By substituting $\omega_i^j=\omega_{i,0}^j+\omega_{i,1}^j \epsilon$ for all $i=1,2$ and $j=1,\ldots,5$ we see that the first line of \eqref{luc11c} reads
\[
[1,\omega^{1}_{1,1}, \omega^1_{2,1},\omega^{1}_{1,0}, \omega^1_{2,0}] 
\]
so  \eqref{luc11c} up to a permutation of columns gives \eqref{luc11b}. One obtains again a bijection between the pairs in $\Omega$ and the points of an arc in $AG(4,q)$. We conclude that $S=|\Omega|\leq q$ for $q$ even.

As the only $(q+1)$-arcs of $\PG(4,q)$ are equivalent to the twisted
cubic (see \cite{BL}), we can put 
\[ \Omega:=\{ ( t+\varepsilon t^2, t^3+\varepsilon t^4 )\colon t\in\GF{q} \}.
  \]

Now, consider the following forms in $4$ indeterminates associated to a BM quasi-Hermitian varieties of $\PG(3,q^2)$
\[F_i(X_0,X_1,X_2,X_3)=X_0^qX_3^q-X_3X_0^{2q-1}+a^q(X_1^{2q}+X_{2}^{2q})-aX_0^{2q-2}(X_1^2+X_{2}^2)+\]	
		\[ -(b^q-b)X_0^{q-1}(X_1^{q+1}+X_{2}^{q+1})+[2a^q(\omega_1^i)^q-(b^q-b)\omega_1^i ]X_0^q X_1^q +[2a^q(\omega_{2}^i)^q-(b^q-b)\omega_{2}^i]X_0^qX_{2}^q+ \]
		\[-[2a\omega_1^i+(b^q-b)(\omega_1^i)^q]X_0^{2q-1}X_1 -[2a\omega_{2}^i+(b^q-b)(\omega_{2}^i)^q]X_0^{2q-1}X_{2},\]
as $\omega_j^i \in \Omega$ for $j=1,2$ and $i=1,\ldots, q $.

\begin{lemma}
  \label{main-code}
  Suppose $q>4$.
  The code \[C(F_1,\dots,F_q;\cW):= \{(F_1(1,x,y,z), F_2(1,x,y,z),\ldots,F_q(1,x,y,z))|(x,y,z) \in \cW\}\]
is an $\GF{q}$-linear $[q,5,q-4]$-MDS code over $T_0$.
\end{lemma}

\begin{proof}
First, we observe that by Lemma~\ref{different}, $|C|=q^5$.
Next, we are going to show that $C$ is $\GF{q}$-linear that is, a vector subspace of $T_0^q$. Take $(x_1,y_1,z_1)$, and $(x_2,y_2,z_2)\in \cW$.

For any $\omega_1, \omega_2 \in \Omega$ set \[F_{\omega_1,\omega_2}=X_0^qX_3^q-X_3X_0^{2q-1}+a^q(X_1^{2q}+X_{2}^{2q})-aX_0^{2q-2}(X_1^2+X_{2}^2)+\]	
		\[ -(b^q-b)X_0^{q-1}(X_1^{q+1}+X_{2}^{q+1})+[2a^q(\omega_1^i)^q-(b^q-b)\omega_1^i ]X_0^q X_1^q +[2a^q(\omega_{2}^i)^q-(b^q-b)\omega_{2}^i]X_0^qX_{2}^q+ \]
		\[-[2a\omega_1^i+(b^q-b)(\omega_1^i)^q]X_0^{2q-1}X_1 -[2a\omega_{2}^i+(b^q-b)(\omega_{2}^i)^q]X_0^{2q-1}X_{2}\]
 Then $F_{\omega_1,\omega_2}(1,x_1,y_1,z_1)+F_{\omega_1,\omega_2}(1,x_2,y_2,z_2)=F_{\omega_1,\omega_2}(1,x_3,y_3,z_3)$
 where $x_3=x_1+x_2$, $y_3=y_1+y_2$ and $z_3=z_1+z_2+2a(x_1x_2+y_1y_2)+(b^q-b)(x_1^qx_2+y_1^qy_2)+s$ with $s \in \GF{q}$ such that $z_3 \in \cC$.

 Furthermore, for any $\lambda \in \GF{q}$ \[\lambda F_{\omega_1,\omega_2}(1,x_1,y_1,z_1)=F_{\omega_1,\omega_2}(1,x,y,z)\]
 where $x=\lambda x_1$, $y=\lambda y_1$ and $z$ is the unique element in $\cC$ which is a root of the equation
 $Z^q-Z=d$, with \[d=\lambda z_1^q-\lambda z_1+a^q(\lambda-\lambda^2)(x_1^{2q}+y_1^{2q})-a(\lambda-\lambda^2)(x_1^2+y_1^2)+(b-b^q)((\lambda-\lambda^2)(x_1^{q+1}+y_1^{q+1}).\]

Since $C$ is $\GF{q}$-linear, from $|C|=q^5$ it follows that the dimension of $C$,
regarded as an $\GF{q}$-vector space, is $5$.

Now, we compute its minimum distance, which is $q-n$ where
\[n=\max_{c\in C\setminus \{\bf{0}\}}|\{i:c_i=0\}|.\]
Because of the Singleton bound, $n\geq 4$. We are going to show that $n=4$. Consider the following system:
\[ \left\{\begin{array}{lllll@{=}l}
 A&+\alpha_1^qB &+\alpha_2^qC&+\alpha_1D&+\alpha_2E&0 \\
 A&+\beta_1^qB &+\beta_2^qC&+\beta_1D&+\beta_2E&0\\
 A&+\gamma_1^qB &+\gamma_2^qC&+\gamma_1D&+\gamma_2E&0\\
 A&+\delta_1^qB &+\delta_2^qC&+\delta_1D&+\delta_2E&0\\
 A&+\eta_1^qB &+\eta_2^qC&+\eta_1D&+\eta_2E&0
     \end{array}\right.\]
for $\alpha_i,\beta_i,\gamma_i,\delta_i,\eta_i \in \Omega$ and \[A=X_3^q-X_3+a^q(X_1^{2q}+X_{2}^{2q})-a(X_1^2+X_{2}^2)-(b^q-b)(X_1^{q+1}+X_{2}^{q+1}),\]
\[B=2a^qX_1^q -(b^q-b)X_1,\]
\[C=2a^qX_2^q -(b^q-b)X_2,\]
\[D=(b-b^q)X_1^q-2aX_1,\]
\[E=(b-b^q)X_2^q-2aX_2.
\]

Condition~\eqref{luc1} implies that the previous system has only one solution $(A,B,C,D,E)=(0,0,0,0,0)$. Under our hypothesis,  we obtain
$X_1=X_2=0$ and $X_3^q-X_3=0$ which have only one solution in $\cC$. Therefore the following system in $X_1$, $X_2$, $X_3$

\[ \left\{\begin{array}{c@{=}l}
F_{\alpha_1,\alpha_2}(1,X_1,X_2,X_3)&0\\
F_{\beta_1,\beta_2}(1,X_1,X_2,X_3)&0\\
F_{\gamma_1,\gamma_2}(1,X_1,X_2,X_3)&0\\
F_{\delta_1,\delta_2}(1,X_1,X_2,X_3)&0\\
F_{\eta_1,\eta_2}(1,X_1,X_2,X_3)&0
 \end{array}\right.\]
 has only one solution in $\cW$, which is $(X_1,X_2,X_3)=(0,0,0)$. This implies that a codeword having at least five zero coordinates is the zero vector, hence $n=4$ and the minimum distance of $C$ is $q-4$.
\end{proof}
It is now convenient to reduce the code $C$ to a code $C'$ which is
defined over $\GF{q}$ instead of $T_0$ ( $T_0$ is isomorphic to $\GF{q}$). Write 
$\varepsilon^q+\varepsilon=a_0$, so, for any $X\in\GF{q^2}$
$\mathrm{Tr}(x_0+\varepsilon x_1)=2x_0+a_0x_1$
So $T_0=\{ x\in\GF{q^2} | \mathrm{Tr}(x)=0 \}=
\{ (a_0-2\varepsilon)x_0 | x_0\in\GF{q} \}$ (observe that
this expression makes sense also in the case $a_0=0$).
Put $\theta=(a_0-2\varepsilon)$.

If $q$ is even, then $T_0=\GF{q}$
and all codewords are in $\GF{q}^n$.
In particular $(\theta)^{-1}C$ is a code over $\GF{q}$.
If $q$ is odd and
we divide each entry of each word in $C$ by $\theta$ and
we obtain a code $C'=(\theta^{-1})C$ (which is obviously equivalent to $C$)
defined over $\GF{q}$.

\begin{theorem}
  Assume $q>4$; then
  the code $C':=(\theta^{-1})C(F_0,F_1,\dots,F_{q-1};\cW)$
  is equivalent to an extended Reed-Solomon code.
  In particular, it can be further extended to a $[q+1,5,q-3]$-code.
\end{theorem}
\begin{proof}
  By Lemma~\ref{main-code}, $C'$ is a $\GF{q}$-linear $[q,5,q-4]$ code
  defined over $\GF{q}$.

  Recall that the polynomials $F_s(1,X_1,X_2,X_3)=F_{\alpha,\beta}(1,X_1,X_2,X_3)$ 
  are defined in terms of coefficients $\alpha,\beta$ where $\alpha=\alpha_0+\varepsilon\alpha_1$
  and $\beta=\beta_0+\varepsilon\beta_1$ and $[1,\alpha_0,\alpha_1,\beta_0,\beta_1]$
  are the coordinates of a point on the standard
  rational normal curve in $\PG(4,q)$.
   Thus, we can take
   \[ (1,\alpha_0,\alpha_1,\beta_0,\beta_1)=
   \begin{cases}
   (1,\omega^s,\omega^{2s},\omega^{3s},\omega^{4s}) & \text{ if $0<s<q$}  \\
   (1,0,0,0,0) & \text{ if $s=0$ }
   \end{cases}
   \]
  for $\omega$ a fixed primitive element of $\GF{q}$.
  So, define the map
  $\psi:\{0,1,\dots,q-1\}\to\GF{q}$ as the function
  \[ \psi(i):=\begin{cases}
      \omega^i & \text{if } i\neq0 \\
      0        & \text{if } i=0
  \end{cases}\]
  and put, for any $t\in\GF{q}$,
   \[ \widetilde{F}_{X_1,X_2,X_3}(t):=F_{\psi^{-1}(t)} (1,X_1,X_2,X_3). \]
Now, regard
  the coefficients $X_1,X_2,X_3$ as a set of  parameters in $\GF{q}$, by writing $X_1=x_0+\varepsilon x_1$, $X_2=y_0+\varepsilon y_1$,
$X_3=z_0+\varepsilon z_1$ with
  $x_i,y_i,z_i\in\GF{q}$.

  We observe that $X_3^q-X_3$ gives $-\theta z_1$ and does not depend on $z_0$. Thus, instead of $\widetilde{F}_{x_0+\varepsilon x_1,y_0+\varepsilon y_1,z_0+\varepsilon z_1}(t)$ we can write $\widetilde{F}_{x_0+\varepsilon x_1,y_0+\varepsilon y_1, \varepsilon z_1}(t)$ and we take into account all of the polynomials.

  Now, for each value of $(x_0,x_1,y_0,y_1,z_1)\in\GF{q}^5$, we can consider the
  polynomials
  \[ \widehat{F}_{x_0,x_1,y_0,y_1,z_1}(t):=
  \theta^{-1}\widetilde{F}_{x_0+\varepsilon x_1,y_0+\varepsilon y_1,\varepsilon z_1}(t).\]
 
  The map $\Psi:(x_0,x_1,y_0,y_1,z_1)\to \widehat{F}_{x_0,x_1,y_0,y_1,z_1}(t)$
  is injective. Indeed,
  suppose that there are two vectors $(x_0,x_1,y_0,y_1,z_1)$
  and $(x_0',x_1',y_0',y_1',z_1')$ providing the same polynomial in $t$.
  This equivalent to say that there are two points 
  $P=(1,x_0+\varepsilon x_1,y_0+\varepsilon y_1,\varepsilon z_1)$ and
  $P'=(1,x_0'+\varepsilon x_1',y_0'+\varepsilon y_1',\varepsilon z_1')$
  such that $P\neq P'$ and for all $i=0,\dots q-1$ we have $F_i(P)=F_i(P')$.
  However, using Condition~\eqref{luc1} as in Lemma~\ref{main-code}, this
  yields a contradiction.
  
 Next, observe that the polynomial $\widehat{F}(t)$, by construction, has degree at most $4$ in $t$.
  Also, since $\widetilde{F}_{X_1,X_2,X_3}(t)\in T_0$, we have
  $\widehat{F}(t)\in\GF{q}$
  for all
  $t\in\GF{q}$; so $\widehat{F}(t)\in\GF{q}[t]$.

The vector space $\GF{q}[t]_{\deg\leq 4}$
of the polynomials of degree at most $4$ in one indeterminate over $\GF{q}$ has
  dimension $5$. Since the image of $\Psi$ has size
  $q^5$, the map $\Psi$ is also surjective and
  the evaluation over $\GF{q}$
  of any polynomial in $\GF{q}[t]_{\deg\leq 4}$
  appears as a codeword in $C'$.
  It follows that, up to a permutation of columns,
  $C=\{ (f(0), f(\omega^1),\dots, f(\omega^{q-1})):
  f\in\GF{q}[t]_{\deg\leq 4} \}$, where $\omega$ is a primitive
  element of $\GF{q}$. This proves that $C'$ is equivalent to an extended Reed-Solomon
  code with parameters $[q,5,q-4]$.
\end{proof}
\begin{remark}

The doubly extended Reed-Solomon code of parameters
$[q+1,5,q-3]$ containing $C'$ is obtained from $C'$ by adding to each codeword a further component with value $F_{q+1}(1,x,y,z)$, where $F_{q+1}(X_0,X_1,X_2,X_3)=(b^q-b)X_2^qX_0+2aX_2X_0^q$.
\end{remark}

\section{Orthogonal arrays}
\label{OA}
In \Cref{orth-intro} we introduced orthogonal arrays and described a method to build one starting from homogeneous forms. In this section we apply the aforementioned method to obtain an orthogonal array from a family of quasi-Hermitian varieties. We keep all previous notation introduced in \Cref{sec:quasiH_varieties}.  The proof of the following theorem relies on \Cref{different}.

\begin{theorem}
Given any prime power $q$, the matrix $\mathcal{A}=A(F^g,g\in \mathcal{R};\mathcal{W})$, where
\[
\mathcal{W}=\{(x_0,\ldots,x_{n})\in \GF{q^2}^{n+1}:x_{0}=1,x_n\in \mathcal{C}\},
\]
is a simple $OA(q^{2n-1},q^{2n-2},q,2)$ of index $\mu=q^{2n-3}$.
\end{theorem}

\begin{proof}
In order to prove the claim we start by studying the system
\begin{equation}\label{eq:system}
\begin{cases}
F(X_0,\ldots,X_n)=\gamma\\
F^g(X_0,\ldots,X_n)=\delta,
\end{cases}
\end{equation}
where $g\in \mathcal{R}\setminus \{Id\}$ and $\gamma, \delta \in T_0$. By definition of $\mathcal{W}$ we can write \eqref{eq:system} as
\[
\begin{cases}\label{eq:system1}
\begin{aligned}
  X_n^q-X_n+a^q(X_1^{2q}+\ldots&+X_{n-1}^{2q})-a(X_1^2+\ldots+X_{n-1}^2)\\-
  &(b^q-b)(X_1^{q+1}+\ldots+X_{n-1}^{q+1})=\gamma
\end{aligned}
  \\ \\
  \begin{aligned}
X_{n}^q-X_n+&a^q(X_1^{2q}+\ldots+X_{n-1}^{2q}+2\alpha_1^qX_1^q+\ldots+2\alpha_{n-1}^qX_{n-1}^q)\\
&-a(X_1^2+\ldots+X_{n-1}+2\alpha_1X_1+\ldots+2\alpha_{n-1}X_{n-1})\\
    &-(b^q-b)\left(\sum_{i=1}^{n-1}X_i^{q+1}+\alpha_i X_i^q+\alpha_i^qX_i\right)=\delta.
    \end{aligned}
\end{cases}
\]
The subtraction of the first equation from the second one gives:
\[
  \begin{cases}\label{eq:system2}
    \begin{aligned}
X_n^q-X_n+a^q\left(\sum_{i=1}^{n-1}X_i^{2q}\right)&-a\left(\sum_{i=1}^{n-1}X_i^2\right)-
(b^q-b)\left(\sum_{i=1}^{n-1}X_i^{q+1}\right)=\gamma
\end{aligned}
    \\ 
    \begin{aligned}
      2a^q\left(\sum_{i=1}^{n-1}\alpha_i^qX_i^q\right)& -2a\left(\sum_{i=1}^{n-1}\alpha_iX_i\right)
      -(b-b^q)\left(\sum_{i=1}^{n-1}\alpha_i X_i^q+\alpha_i^qX_i\right)=\delta-\gamma.
      \end{aligned}
\end{cases}
\]
Let $A:=-2a(\sum_{i=1}^{n-1}\alpha_iX_i)-b(\sum_{i=1}^{n-1}\alpha_i^qX_i)$; we notice that the second equation of the system above is of the kind
\begin{equation}\label{eq:A_polynomial}
A^q-A=\delta-\gamma.
\end{equation}
By~\cite{H1}*{Theorem 1.19}, this equation has $q$ solutions if $Tr(\delta-\gamma)=0$. The latter is always true because $\delta,\gamma\in T_0$. We conclude that there are $q$ possible values of $A$ that satisfy \eqref{eq:A_polynomial}, therefore there are $q^{2n-3}$ $(n-1)$-tuples $(X_1,\dots,X_{n-1})\in \GF{q^2}^{n-1}$ that satisfy the system of equations. Recalling that $X_n$ ranges over $\mathcal{C}$, we deduce that the number of solutions of~\eqref{eq:system} is $q^{2n-3}$.

The fact that the array is simple, i.e. that it does not contain repeated rows, 
now follows from Lemma~\ref{different}.
 \end{proof}

\section*{Acknowledgements}
All of the authors thank the Italian National Group for Algebraic and Geometric Structures and their Applications (GNSAGA--INdAM) for its support to their research.
The work of A. Aguglia and V. Siconolfi has also been
 partially founded by the European Union under the Italian National Recovery and Resilience Plan (NRRP) of NextGenerationEU, partnership on “Telecommunications of the Future” (PE00000001 - program ‘‘RESTART’’, CUP: D93C22000910001).


\end{document}